\documentclass[10pt]{article}
\usepackage{amsmath,amssymb,amsthm}
\usepackage{amsfonts}
\usepackage{enumerate}
\usepackage[single]{accents}
\usepackage[margin=1.5in]{geometry}
\usepackage{color}
\usepackage{hyperref}
\usepackage{graphicx}
\usepackage{float}
\usepackage{tikz}
\usetikzlibrary{shapes}
\usetikzlibrary{plotmarks}
\usetikzlibrary{arrows}
\usetikzlibrary{positioning}
\definecolor{c0000FF}{RGB}{0,0,255}

\def\a#1{\mathfrak{#1}}

\theoremstyle{definition}
\newtheorem{defn}{Definition}[section]
\theoremstyle{plain}
\newtheorem{thm}[defn]{Theorem}
\newtheorem{lem}[defn]{Lemma}
\usepackage{authblk}
\usepackage{tikz}
\title{General normal forms for any additive logic}
\author{Mohamed Khaled\thanks{khaled.mohamed@renyi.mta.hu}}
\affil{Alfr\'ed R\'enyi Institute of Mathematics, Hungarian Academy of Sciences, Budapest, Hungary}
\affil{Department of Mathematics, Faculty of Science, Cairo University, Giza, Egypt}
\date{}
\newcommand\myeq{\mathrel{\stackrel{\makebox[0pt]{\mbox{\normalfont\tiny def}}}{=}}}

\begin{document}
\maketitle
\begin{abstract}
In this article, we define general normal forms for any logic that has propositional part and whose non-propositional connectives distribute over the finite disjunctions. We do not require the non-propositional connectives to be closed on the set of formulas, so our normal forms cover logics with partial connectives too. We also show that most of the known normal forms in the literature are in fact particular cases of our general forms. These general normal forms are natural improvement of the distributive normal forms of J. Hintikka \cite{hintikka} and their modal analogues, e.g. \cite{anderson} and \cite{fine}.
\end{abstract}
\section{Introduction}
It was shown that every propositional formula can be rewritten equivalently in the disjunctive normal form. Such form is a disjunction of one or more conjunctive clauses, each of these clauses consists of statement letters or negations of statement letters. Similar forms were introduced for more complex logics, e.g. distributive normal forms for first order logic \cite{hintikka}. Normal forms were also extended to some modal logics and then they were used to give elegant and constructive proofs for many standard results, see \cite{anderson} and \cite{fine}. Recently, the method of normal forms was used in new directions to solve problems where all the standard techniques fail, e.g. \cite{van}, \cite{myrsl}, \cite{myphd} and \cite{myigpl}. 

Here, we generalize the normal forms in a novel way. We prove a general theorem that can be applied for arbitrary logics that extend the propositional one. But, we also require the non-propositional connectives to be distributive over the finite disjunctions. For each $k\in\omega$, we define a set of formulas that we call {\underline{normal forms of degree $k$}}. Each one of these forms is a conjunction of (constituents consist of) any one of the following or their negations: propositional letters, and normal forms of the first smaller degree joined together by a non-propositional connective. That is, we build a hierarchy of normal forms by considering their connections, via the non-propositional connectives, with the normal forms of the first smaller degree. Then we prove that any formula of the logic in question can be rewritten in an equivalent way as a disjunction of finitely many formulas that live in the same level of the hierarchy of the normal forms.

Having these degrees, one can view each normal form $\varphi$ of degree $k$ as a hierarchy in itself. This hierarchy consists of $k$-levels, its first level consists of a unique element that represents $\varphi$. Then, to go up one level, one needs to consider brand new elements to represent the normal forms that appeared together joined by a non-propositional connective and without negation symbol in the conjuncts of $\psi$, where $\psi$ is a normal form represented by an element in the current level of the hierarchy. Thus, this hierarchy will end with a level that contains elements represent forms of degree $0$ only. Such hierarchy (or its modified versions) was shown, in most of the cases, to be a fair witness for the validity of the normal form $\varphi$.

Thus, since such hierarchies are finite, one may obtain some decidability results through the finite model property \cite{fine,myigpl}. Also these hierarchies can provide easy and transparent proofs for some completeness results \cite{fine}. Moreover, these finite hierarchies associated with the normal forms can be extended by adding more levels. With a careful choice of such extensions one might be able to investigate the so-called G\"odel's incompleteness properties, e.g. \cite{myigpl}. These properties talk about the possibilities of refining the definable concepts by adding more conjuncts to the formulas defining them. These properties were defined by I. N\'emeti in 1985 and named by him in appreciation of the pioneering work of K. G\"odel. For more details concerning these incompleteness properties, see \cite{nem86,zalan,myphd,myrsl,myigpl}.

Normal forms are essentially used in establishing results in both algebra and logic, e.g. \cite{sagi99}, \cite{sagi01} and \cite{sagi02}. In the application section, we apply our general theorem to produce normal forms for Boolean algebras with operators. These `algebraic' normal forms were used recently to give an answer for a problem that was raised by I. N\'emeti in 1985. Indeed, the non-atomicity of the free algebras of several classes of algebras of logics was shown (c.f. \cite{myphd}, \cite{myrsl} and \cite{myigpl}). Moreover, normal forms were also used to give a complete description for some of these free algebras from the point of view of atoms \cite{myphd}.

One more feature of our result is, according to our definition of arbitrary additive logics, that the non-propositional connectives can be partial operators on the set of formulas, i.e., they do not need to be closed. That allows us to apply our normal forms to several important logics, e.g. guarded fragment of first order logic \cite{andvannem}, the loosely guarded fragment \cite{vB97} and the packed fragment \cite{mar01}. Those are versions of first order logic in which the quantifiers are bounded and can not be applied to any formula. These logics were investigated by many logicians and it was shown that they have a number of desirable properties, e.g. decidability and finite model property. They have applications in linguistics (dynamic semantics of natural language) and computer science.
\section{Additive logics}
In the present paper, an arbitrary logic is a tuple $\mathcal{L}=\langle Fm,\Vdash\rangle$, where $Fm$ is the set of formulas of $\mathcal{L}$ and $\Vdash$ is a unary relation on the set of formulas. Instead of writing $\varphi\in\Vdash$, we rather write $\Vdash\varphi$, for any formula $\varphi\in Fm$. This unary relation could be the set of provable formulas, the set of valid formulas, or may be something else. Let $P$ be a non-empty set of propositions and let $Cn$ be a set of connectives each of which is associated with a non-zero finite rank. We say that $\mathcal{L}$ is generated by $P$ and $Cn$ if $Fm$ is built up recursively from the set $P$ using the connectives in $Cn$, i.e.,
\begin{enumerate}[]
\item A string $\varphi$ of propositions and connectives is a formula of $\mathcal{L}$ only if $\varphi\in P$ or $\varphi$ is of the form
$\varphi=\Box(\psi_0,\ldots,\psi_{k-1})$ for some $\Box\in Cn$ whose rank is $k$ and some $\psi_0,\ldots,\psi_{k-1}\in Fm$.
\end{enumerate}
Suppose that $\mathcal{L}$ is generated by $P$ and $Cn$. Let $\Box\in Cn$ be a connective, say of rank $k$, and let $\psi_0,\ldots,\psi_{k-1}\in Fm$. Note that $\Box(\psi_0,\ldots,\psi_{k-1})$ is not necessarily a formula of $\mathcal{L}$. Define $$D(\Box)\myeq\{(\psi_0,\ldots,\psi_{k-1})\in{^kFm}:\Box(\psi_0,\ldots,\psi_{k-1})\text{ is a formula in }Fm\}.$$
We say that $\mathcal{L}$ has a propositional part if, in addition to above, it satisfies the following:
\begin{enumerate}[(P1)]
\item The usual propositional connectives $\lor$, $\land$ and $\lnot$ are members of $Cn$. Moreover, for any formulas $\varphi,\psi\in Fm$,
\begin{itemize}
\item $\lnot\varphi\myeq\lnot(\varphi)$ is a formula in $Fm$.
\item $\varphi\lor\psi\myeq\lor(\varphi,\psi)$ is a formula in $Fm$.
\item $\varphi\land\psi\myeq\land(\varphi,\psi)$ is a formula in $Fm$.
\end{itemize}
In other words, $D(\lnot)=Fm$, $D(\lor)={^2Fm}$ and $D(\land)={^2Fm}$. As usual, we define the derived connectives $\varphi\rightarrow\psi\myeq\lnot\varphi\lor\psi$ and $\varphi\leftrightarrow\psi\myeq(\varphi\rightarrow\psi)\land(\psi\rightarrow\varphi)$.
\item For each $\Box\in Cn$ whose rank is $h$ and $\varphi_0,\psi_0,\ldots,\varphi_{h-1},\psi_{h-1}\in Fm$,
\begin{eqnarray*}
&&\text{if }(\varphi_0,\ldots,\varphi_{h-1})\in D(\Box)\ \text{ and }\ (\psi_0,\ldots,\psi_{h-1})\in D(\Box)\text{ then }\ \ \ \ \ \ \ \ \ \ \ \ \ \ \ \ \ \ \\
&&\ \ \ \ \ \ \ \ \ \ \Big[\bigwedge\{\Vdash \varphi_i\longleftrightarrow\psi_i: i\in h\} \implies \Vdash \Box(\varphi_0,\ldots,\varphi_{h-1})\longleftrightarrow \Box(\psi_0,\ldots,\psi_{h-1})\Big].
\end{eqnarray*}
\item For each non-propositional connective $\Box\in Cn'\myeq Cn\setminus\{\lnot,\lor,\land\}$, say of rank $h$, and formulas $\varphi_0,\ldots,\varphi_{h-1}\in Fm$ with $(\varphi_0,\ldots,\varphi_{h-1})\in D(\Box)$,
    $$\bigvee\{\Vdash \lnot\varphi_i: i\in h\} \implies\Vdash\lnot\Box(\varphi_0,\ldots,\varphi_{h-1}).\footnote{This is a normality condition, ituitively it says that applying non-propostional connectives to a contradiction will give a contradiciton again.}$$
\item The propositional tautologies are members of the unary relation $\Vdash$.
\end{enumerate}
The above conditions are indeed relevant to what is meant in the literature for a logic to have a propositional part. Now, we define what we mean by an additive logic.
\begin{defn}\label{additive}
An additive logic $\mathcal{L}=\langle Fm, \Vdash\rangle$ is an arbitrary logic that has propositional part and satisfies the following condition:
\begin{enumerate}[(Add)]
\item For any non-propositional connective $\Box\in Cn'$, say of rank $k$, any $i\in k$ and any formulas $\varphi_0,\ldots,\varphi_{i-1},\varphi,\psi,\varphi_{i+1},\ldots,\varphi_{k-1}$, if $(\varphi_0,\ldots,\varphi_{i-1},\varphi\lor\psi,\varphi_{i+1},\ldots,\varphi_{k-1})$, $(\varphi_0,\ldots,\varphi_{i-1},\varphi,\varphi_{i+1},\ldots,\varphi_{k-1})$ and $(\varphi_0,\ldots,\varphi_{i-1},\psi,\varphi_{i+1},\ldots,\varphi_{k-1})$ are all in the domain $D(\Box)$ then
    \begin{eqnarray*}
   \Vdash\Box(\varphi_0,\ldots,\varphi_i,\varphi\lor\psi,\varphi_{i+1},\ldots,\varphi_{k-1})\leftrightarrow && \Box(\varphi_0,\ldots,\varphi_{i-1},\varphi,\varphi_{i+1},\ldots, \varphi_{k-1})\lor\\
   &&\Box(\varphi_0,\ldots,\varphi_{i-1},\psi,\varphi_{i+1},\ldots,\varphi_{k-1})
    \end{eqnarray*}
\end{enumerate}
\end{defn}
Most of the important logics, together with their set of provable formulas or their set of valid formulas, are in fact additive, e.g. propositional logic, first order logic (with or without equality), guarded fragment of first order logic \cite{andvannem} and its liberal versions \cite{vB97,mar01}, non-permutable first order logic \cite{richard,nem86}, first  order logic with general assignment models \cite{nem86}, modal logics, etc. However, there are few examples of logics that are not additive: logics that do not have propositional part, e.g. intuitionistic logic, and logics that have propositional part but they are not additive, e.g. instantial neighbourhood logic \cite{van}.
\begin{defn}\label{system}
Let $\mathcal{L}=\langle Fm, \Vdash\rangle$ be an additive logic. We say that $\mathcal{L}$ has a domain-representation system if there is a quadrable $(V,\iota,\jmath_1,\jmath_2)$, where $V$ is an arbitrary non-empty set, and $\iota:Fm\rightarrow\mathcal{P}(V)$ and $\jmath_1,\jmath_2:Cn'\rightarrow\mathcal{P}(V)$ are arbitrary maps such that the following conditions are true for any connective $\Box \in Cn'$ of rank $h\in\omega$ and any formulas $\varphi,\psi,\varphi_0,\ldots,\varphi_{h-1}\in Fm$:
\begin{enumerate}[(a)]
\item $\iota(\lnot\varphi)=\iota(\varphi)$, $\iota(\varphi\land\psi)=\iota(\varphi)\cup\iota(\psi)$ and $\iota(\varphi\lor\psi)=\iota(\varphi)\cup\iota(\psi)$.
\item $\iota(\Box(\varphi_0,\ldots, \varphi_{h-1}))=\jmath_2(\Box)\setminus\jmath_1(\Box)$.
\item\label{c} $(\varphi_0,\ldots,\varphi_{h-1})\in D(\Box)\iff \bigwedge\{\iota(\varphi_k)\subseteq\jmath_2(\Box):k\in h\}$.
\end{enumerate}
\end{defn}
Suppose that every non-propositional connective is unary. The domain-representation system is acting as a ``chart of domain areas'' that tells us which formula is in the domain of which connective of $Cn'$. The set $V$ is a broadsheet paper and $\jmath_2$ draws the borders of each non-propositional connective on that paper. The formulas are also represented via $\iota$ as bounded figures on the chart. According to condition (c) above, the bounded figure represents a formula $\varphi$ has to be inside the borders of a non-propositional connective $\Box$ if and only if $\varphi\in D(\Box)$ (See Figure~\ref{fig}). Condition (b) indicates that $\jmath_1$ is there only to represent any restrictions appear on the formula $\Box(\varphi)$ (after applying the connective $\Box$).

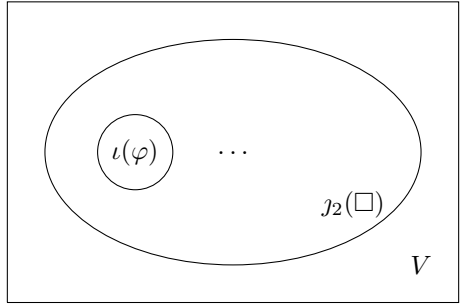
\begin{figure}[!h]
\centering
\begin{tikzpicture}
\centering    
\draw (-3,-2) -- (-3,2) -- (3,2) -- (3,-2) -- (-3,-2);
\node at (2.5,-1.5) {$V$};
\draw (0,0) ellipse (2.5cm and 1.5cm);
\node at (1.6,-0.7) {$\jmath_2(\Box)$};
\draw (-1.3,0) circle (0.5cm);
\node at (-1.3,0) {$\iota(\varphi)$};
\node at (0,0) {$\ldots$};
\end{tikzpicture}
\caption{Domain-representation systems}
\label{fig}
\end{figure}

We will construct normal forms for any additive logic associated with a fixed domain-representation system. We note that there might be more than one domain-representation system for the same logic. Definition \ref{system} is liberal enough to guarantee the existence of domain-representation systems for all the additive logics mentioned above. In fact, we do not know any example of an additive logic that does not have such system\footnote{We are aware that such an example can be constructed. What we meant here is that additive logics with no domain-representation systems have not been studied yet in mathematical logic.}. For example, if all the non-propositional connectives in an additive logic $\mathcal{L}$ are full operators, then we can define a domain-representation system as follows. Let $V\not=\emptyset$ be an arbitrary set. For each formula $\varphi$ of $\mathcal{L}$, define $\iota(\varphi)=V$. For each non-propositional connective $\Box$, define $\jmath_1(\Box)=\emptyset$ and $\jmath_2(\Box)=V$. It is easy to see that $(V,\iota,\jmath_1,\jmath_2)$ satisfies all the above conditions.
\section{Main result}
Fix an additive logic $\mathcal{L}=\langle Fm,\Vdash\rangle$ and a domain-representation system $(V,\iota,\jmath_1,\jmath_2)$. We note that the normal forms we construct here depend on the choice of this system. Some domain-representation systems may be suitable for a certain task, some may not. Let $\Sigma$ be a non-empty finite set of formulas and fix any enumeration of $\Sigma=\{\varphi_0,\ldots,\varphi_{k-1}\}$, where $k$ is the cardinality of the set $\Sigma$. Define 
$\bigwedge\Sigma$ and 
$\bigvee\Sigma$ as follows:
$$\bigwedge\Sigma\myeq\varphi_0\land\cdots\land\varphi_{k-1} \ \text{ and } \ \bigvee\Sigma\myeq\varphi_0\lor\cdots\lor\varphi_{k-1}.$$
This might seem ambiguous because the definitions of $\bigwedge\Sigma$ and $\bigvee\Sigma$ depend on the chosen enumeration for $\Sigma$. But, by condition (P4), one can see that considering any other enumeration will yield to an equivalent formula, up to the relation $\Vdash$. A function $\alpha\in{^{\Sigma}\{-1,1\}}$ allows us to choose, for each formula in $\Sigma$, either the formula itself or its negation, i.e. we define $\varphi^{\alpha}\myeq\varphi$ if $\alpha(\varphi)=1$ and $\varphi^{\alpha}\myeq\lnot\varphi$ otherwise. Then, we can define the formula $$\Sigma^{\alpha}\myeq\bigwedge\{\varphi^{\alpha}:\varphi\in\Sigma\}.$$

We need more conventions. Recall that $Cn'=Cn\setminus\{\lnot,\lor,\land\}$. Let $X\subseteq P$ be a finite set and let $A\subseteq V$. Define $\widetilde{X}\myeq\{p\in X:\iota(p)\subseteq A\}$. We say that $A$ is large enough for $X$ if $\widetilde{X}\not=\emptyset$. Let $\Box\in Cn'$, we say that $\Box$ is compatible with $(X,A)$ if $\jmath_2(\Box)\setminus\jmath_1(\Box)\subseteq A$ and $\jmath_2(\Box)$ is large enough for $X$. Let $Y\subseteq Cn'$ be finite. Now, we are ready to define our normal forms. For each $k\in\omega$, we will build a set $N_k(X,Y;A)$ of normal forms of degree $k$. Each $\varphi\in N_k(X,Y;A)$ is built up from the propositions in $X$, the Boolean connectives and the connectives in $Y$. Throughout the construction, we restrict ourself to the condition that $\iota(\varphi)\subseteq A$, for each $\varphi\in N_k(X,Y;A)$. Then we will prove that any formula of $\mathcal{L}$ can be rewritten in an equivalent form, up to the relation $\Vdash$, as a disjunction of normal forms of the same degree (actually in the same $N_k(X,Y;A)$).
\begin{defn}Let $X\subseteq P$ and $Y\subseteq Cn'$ be finite sets, and let $A\subseteq V$ be large enough for $X$. Let $k\in \omega$, we define the set $N_k(X,Y;A)$ (of normal forms of degree $k$) inductively as follows.
\begin{itemize}
\item Normal forms of degree $0$: $$N_0(X,Y;A)\myeq\{(\widetilde{X})^{\alpha}: \alpha\in{^{\widetilde{X}}\{-1,1\}}\}.$$
\item Normal forms of degree $k+1$: $$N_{k+1}(X,Y;A)\myeq\{(\widetilde{X})^{\alpha}\land (\overline{N_k(X,Y;A)})^{\beta}:\alpha\in{^{\widetilde{X}}\{-1,1\}}, \beta\in{^{\overline{N_k(X,Y;A)}}\{-1,1\}}\},$$
    where $\overline{N_k(X,Y;A)}$ is defined as follows. If there is at least one connective $\Box\in Y$ which is compatible with $(X,A)$, then $\overline{N_k(X,Y;A)}$ consists of all formulas of the form $\diamondsuit(\varphi_0,\ldots,\varphi_{h-1})$, where $\diamondsuit\in Y$ is a connective compatible with $(X,A)$, $h\in\omega$ is the rank of $\diamondsuit$ and $\varphi_0,\ldots,\varphi_{h-1}$ are normal forms in $N_k(X,Y;\jmath_2(\diamondsuit))$ \footnote{This guarantees that $(\varphi_0,\ldots,\varphi_{h-1})\in D(\diamondsuit)$ because of condition \eqref{c} in Definition \ref{system}.}. Otherwise, $\overline{N_k(X,Y;A)}\myeq N_k(X,Y;A)$
\item Finally, define $N(X,Y;A)\myeq\bigcup\{N_k(X,Y;A):k\in\omega\}$.
\end{itemize}
By induction on $k\in\omega$, one can prove that $\overline{N_k(X,Y;A)}\not=\emptyset$, hence the above sets of normal forms are well defined.
\end{defn}
A normal form $\varphi\in N_k(X,Y;A)$ is a formula that is built up from the set $X$ using the connectives in $Y$ together with the propositional connectives such that $k$ is the depth of the non-propositional connectives nesting in $\varphi$. Each normal form carries some information indicating its connections (via the non-propositional connectives) with the normal forms of the first smaller degree. We introduce the following notions that allow us to handle this information easily when it is needed.
\begin{defn}
Let $X\subseteq P$ and $Y\subseteq Cn'$ be finite sets, and let $A\subseteq V$ be large enough for $X$.
\begin{itemize}
\item Let $\varphi\in N_0(X,Y;A)$ be a normal form of degree $0$. Then there is $\alpha\in{^{\widetilde{X}}\{-1,1\}}$ such that $\varphi=(\widetilde{X})^{\alpha}$. Define $color(\varphi)\myeq\{p\in \widetilde{X}: \alpha(p)=1\}$.
\item Let $k\in\omega$ and let $\varphi\in N_{k+1}(X,Y;A)$. Then there are choices $\alpha\in{^{\widetilde{X}}\{-1,1\}}$ and $\beta\in{^{\overline{N_k(X,Y;A)}}\{-1,1\}}$ such that $\varphi=(\widetilde{X})^{\alpha}\land (\overline{N_k(X,Y;A)})^{\beta}$. Define
    $$color(\varphi)\myeq\{p\in \widetilde{X}: \alpha(p)=1\}.$$
    If $\Box\in Y$ is compatible with $(X,A)$ and its rank is $h\in\omega$, then define
    $$sub_{\Box}(\varphi)\myeq\{(\varphi_0,\ldots,\varphi_{h-1})\in {^hN_k(X,Y,\jmath_2(\Box))}:\beta(\Box(\varphi_0,\ldots,\varphi_{h-1}))=1\}.$$
\end{itemize}
\end{defn}
The lemma below follows immediately from the construction of $N_k(X,Y;A)$ and (P4).
\begin{lem}\label{lem}
Let $X\subseteq P$ and $Y\subseteq Cn'$ be finite sets, and let $A\subseteq V$ be large enough for $X$. Let $k\in \omega$, then we have the following: $\Vdash\bigvee N_k(X,Y;A)$, and for every normal forms $\varphi,\psi\in N_k(X,Y;A)$ if $\varphi\not=\psi$ then $\Vdash\lnot(\varphi\land\psi)$.
\end{lem}

Let $\psi\in Fm$ be an arbitrary formula of $\mathcal{L}$ and let $d$ be the maximum depth of the non-propositional connectives nesting in $\psi$. We define $P(\psi)$ and $Cn(\psi)$ to be the set of all propositions and the set of all connectives from $Cn'$, respectively, that appear in $\psi$. Let $k\in\omega$, let $X\subseteq P$ and $Y\subseteq Cn'$ be finite sets, and let $E\subseteq V$. We say that $(k,X,Y,E)$ is a generator suitable for $\psi$ if $k\geq d$, $X\supseteq P(\psi)$, $Y\supseteq Cn(\psi)$ and $E\supseteq\iota(\psi)$, $A$ is large enough for $X$ and, for each $\Box\in Cn(\psi)$, $\jmath_2(\Box)$ is large enough for $X$.
\begin{thm}\label{thm}
Let $\varphi$ be any formula of $\mathcal{L}$. Then, for any generator $(k, X, Y, E)$ suitable for $\varphi$, there is a finite $\Sigma\subseteq N_k(X,Y;E)$ such that $\Vdash\varphi\longleftrightarrow\bigvee\Sigma$\footnote{Here, if $\Sigma=\emptyset$ then we define $\bigvee\Sigma$ to be a contradiction.}. Moreover, there is a finite algorithm (that does not depend on the relation $\Vdash$) which generates such $\Sigma\subseteq N_k(X,Y;E)$ for any $\varphi$ and $(k,X,Y,E)$ as above.
\end{thm}
\begin{proof}
Let $\varphi$ be any formula and let $(k,X,Y,E)$ be a generator suitable for $\varphi$. We use induction on the complexity of $\varphi$.
\begin{itemize}
\item Suppose that $\varphi$ is a proposition. Then we have $k\geq 0$ and $\varphi\in X\subseteq P$. Set
$$\Sigma=\{\psi\in N_k(X,Y;E): \varphi\in color(\psi)\}.$$
Then, by (P4), it is easy to see that $\Vdash\varphi\longleftrightarrow\bigvee\Sigma$.
\item Suppose that $\varphi=\varphi_1\land\varphi_2$ for some formulas $\varphi_1,\varphi_2\in Fm$. Let $j\in\{1,2\}$. Note that $(k,X,Y,E)$ is also a generator suitable for $\varphi_j$. Thus, by induction hypothesis, there is $\Sigma_j\subseteq N_k(X,Y;E)$ such that $\Vdash\varphi_j\longleftrightarrow\bigvee\Sigma_j$. By Lemma \ref{lem}, for each $\psi_1\in \Sigma_1$ and $\psi_2\in\Sigma_2$, if $\psi_1\not=\psi_2$ then $\Vdash \lnot(\psi_1\land\psi_2)$. Hence, by (P4), there is a finite set $\Sigma(=\Sigma_1\cap\Sigma_2)\subseteq N_k(X,Y;E)$ such that
    $$\Vdash\bigvee\Sigma\longleftrightarrow\bigvee\{\psi_1\land \psi_2: \psi_1\in \Sigma_1\text{ and }\psi_2\in\Sigma_2\}.$$
    Now, (P4) implies that $$\Vdash\varphi\longleftrightarrow\bigvee\{\psi_1\land \psi_2: \psi_1\in \Sigma_1\text{ and }\psi_2\in\Sigma_2\}.$$
    Therefore, $\Vdash \varphi\longleftrightarrow\bigvee\Sigma$ as desired. The induction step goes in a similar way for the disjunction $\varphi_1\lor\varphi_2$.
\item Suppose that $\varphi=\lnot\psi$ for some formula $\psi\in Fm$. Again $(k,X,Y,E)$ is a generator suitable for $\psi$. Then, by the induction hypothesis, there is a finite set of normal forms $\Sigma'\subseteq N_k(X,Y;E)$ such that $\Vdash\psi\longleftrightarrow\bigvee\Sigma'$. Let $\Sigma=N_k(X,Y;E)\setminus\Sigma'\subseteq N_k(X,Y;E)$. Therefore, by Lemma \ref{lem} and (P4), we have $\Vdash\varphi\longleftrightarrow\bigvee\Sigma$.
\item Suppose that $\varphi=\diamondsuit(\varphi_0,\ldots,\varphi_{h-1})$ for some $\diamondsuit\in Cn'$ whose rank is $h\in\omega$ and some formulas $\varphi_0,\ldots,\varphi_{h-1}$ such that $(\varphi_0,\ldots,\varphi_{h-1})\in D(\diamondsuit)$. Let $j\in h$. By condition (c) in Definition \ref{system}, we have $\iota(\varphi_j)\subseteq\jmath_2(\diamondsuit)$. Hence, $(k-1,X,Y,\jmath_2(\diamondsuit))$ is a generator suitable for $\varphi_j$. Thus, by induction hypothesis, there is $\Sigma_j\subseteq N_{k-1}(X,Y;\jmath_2(\diamondsuit))$ such that $\Vdash \varphi_j\longleftrightarrow\bigvee\Sigma_j$. Hence, by (P2), $\Vdash \varphi\longleftrightarrow \diamondsuit(\bigvee\Sigma_0,\ldots,\bigvee\Sigma_{h-1})$.
    Then, the additivity of $\mathcal{L}$ (Add) together with (P3) imply
    \begin{equation}\label{eq}\Vdash \varphi\longleftrightarrow \bigvee\{\diamondsuit(\varphi_0,\ldots,\varphi_{h-1}):(\forall j\in h) \varphi_j\in \Sigma_j\}.\tag{*}\end{equation}
    Note that $\jmath_2(\diamondsuit)\setminus\jmath_1(\diamondsuit)=\iota(\varphi)\subseteq A$ and by assumptions $\jmath_2(\diamondsuit)$ is compatible with $(A,X)$. So, by the construction, each disjunct in \eqref{eq} is a member of $\overline{N_{k-1}(X,Y;A)}$. Therefore, (P4) implies that $\Vdash\varphi\longleftrightarrow\bigvee\Sigma$, where
    $$\Sigma=\{\psi\in N_k(X,Y;E):sub_{\diamondsuit}(\psi)\cap (\Sigma_0\times\cdots\times\Sigma_{h-1})\not=\emptyset\}.$$
\end{itemize}
The above steps also write a finite algorithm that generates the set $\Sigma\subseteq N_k(X,Y;E)$ for any $\varphi$ and $(k,X,Y;E)$ as required.
\end{proof}
The above theorem can be explained intuitively as follows. Recall the chart of the domain areas and let $E$ be any part of the paper sheet. Then there are normal forms that are represented on the part $E$ of the chart, and these forms generate any other formula (whose elementary components are those components of the normal forms) represented inside $E$.
\section{Applications}
In this section, we apply Theorem \ref{thm} to several logics and also to Boolean algebras with operators. We start with the first known example of disjunctive normal forms.
\subsection{Propositional logic}Suppose that $\mathcal{L}$ is the propositional logic with set of propositions $P$. Let $\emptyset\not=X\subseteq P$ be a finite set of propositions, define the following set:
$$F(X)\myeq\{X^{\alpha}:\alpha\in{^X\{-1,1\}}\}.$$
\begin{thm}
There is a finite algorithm that generates, for every formula $\varphi$ of the propositional logic $\mathcal{L}$ and every finite $X\subseteq P$ that contains all the propositions appearing in $\varphi$, a finite $\Sigma\subseteq F(X)$ such that $\vdash\varphi\longleftrightarrow\bigvee\Sigma$ and $\models\varphi\longleftrightarrow\bigvee\Sigma$.
\end{thm}
\begin{proof}
Let $V\not=\emptyset$ be any set and define $\iota:Fm\rightarrow\mathcal{P}(V)$ such that $\iota(\varphi)=V$ for every $\varphi\in Fm$. Note that $Cn'$ in this case is empty, so we can choose $\jmath_1=\jmath_2=\emptyset$. Thus, $(V,\iota,\jmath_1,\jmath_2)$ is a domain-representation system for $\mathcal{L}$. Moreover, $(0,X,\emptyset,V)$ is suitable generator for $\varphi$ and $F(X)=N_0(X,\emptyset;V)$. Therefore, we are done by Theorem \ref{thm}.
\end{proof}
The distinction between the provability relation and the validity relation of propositional logic above is superfluous, indeed it is well known that propositional logic is (strongly) sound and complete. Now, we will give more interesting examples. We will construct normal forms for several version of first order logic that keeps the propositional axioms and the distributivity of the existential quantifiers over finite disjunctions\footnote{These normal forms are in fact the distributive normal forms of J. Hintikka \cite{hintikka}}. For convenience, we consider first order language $L$ that contains only variables $VAR$ (we also assume that $\mid VAR\mid\geq 2$) and relation symbols $REL$ (may or may not contain the equality symbol $=$). The connectives are the propositional ones $\lnot,\lor,\land$ together with the existential quantifiers $\exists v$. We will also suppose a relation $\Vdash$ on the set $Fm$ of formulas of the first order formulas (on the language $L$) that makes $\mathcal{L}=\langle Fm,\Vdash\rangle$ additive in the sense of Definition \ref{additive}\footnote{For instance, $\Vdash$ can be the set of provable formulas of the standard or the non permutable first order logic, also $\Vdash$ can be the set of valid formulas of first order logic over standard models or general assignments models.}.
\subsection{First order-like logics}
For any $X\subseteq VAR$ and any $Y\subseteq REL$, let $P(X,Y)$ be the set of all atomic formulas that use variables from $X$ and relation symbols form $Y$. To define the normal forms, we need to restrict ourselves to some finite subset of the language. If $X\subseteq VAR$ and $Y\subseteq REL$ are both finite, then $P(X,Y)$ is also finite.
\begin{defn}
Let $X\subseteq VAR$ and $Y\subseteq REL$ be finite subsets such that $P(X,Y)\not=\emptyset$, and let $k\in\omega$. Define the following inductively:
\begin{enumerate}[-]
\item $F_0(X,Y)\myeq\{P(X,Y)^{\alpha}:\alpha\in{^{P(X,Y)}\{-1,1\}}\}$.
\item $F_{k+1}(X,Y)\myeq\{P(X,Y)^{\alpha}\land(\overline{F_k(X,Y)})^{\beta}:\alpha\in{^{P(X,Y)}\{-1,1\}},\beta\in{^{\overline{F_k(X,Y)}}\{-1,1\}}\}$,

\noindent where if $X=\emptyset$ then $\overline{F_k(X,Y)}=F_k(X,Y)$, otherwise $$\overline{F_k(X,Y)}=\{\exists v \ \varphi: v\in X\text{ and }\varphi\in F_k(X,Y)\}.$$
\end{enumerate}
\end{defn}
\begin{thm}\label{FOL}
There is a finite algorithm that generates, for every first order formula $\varphi$, every $k$ bigger than or equal to the maximum depth of quantifiers nesting in $\varphi$, every finite $X\subseteq VAR$ that contains all the variables appearing in $\varphi$ and every finite $Y\subseteq REL$ that contains all relation symbols appearing in $\varphi$, a finite $\Sigma\subseteq F_k(X,Y)$ such that $\Vdash\varphi\longleftrightarrow\bigvee\Sigma$.
\end{thm}
\begin{proof}
Let $V\not=\emptyset$ be any set. For each formula $\psi$, define $\iota(\psi)=V$. For each $v\in VAR$, define $\jmath_1(\exists v)=\emptyset$ and $\jmath_2(\exists v)=V$. Then, $(V,\iota,\jmath_1,\jmath_2)$ is a domain-representation system. Let $\mathbb{X}$ be the set of all atomic formulas built up from $X$ and $Y$, and let $\mathbb{Y}=\{\exists v:v\in X\}$. Thus, $(k,\mathbb{X},\mathbb{Y},V)$ is a generator suitable for $\varphi$ and $N_k(\mathbb{X},\mathbb{Y};V)=F_k(X,Y)$. Therefore, the statement follows immediately by Theorem \ref{thm}.
\end{proof}
Suppose that $L$ is finite, that is suppose $VAR\cup REL$ is finite. Then the following is an immediate corollary of Theorem \ref{FOL}.
\begin{thm}
There is a finite algorithm that generates, for every first order formula $\varphi$ on the finite language $L$ and every $k$ bigger than or equal to the maximum depth of quantifiers nesting in $\varphi$, a finite $\Sigma\subseteq F_k(VAR,REL)$ such that $\Vdash\varphi\longleftrightarrow\bigvee\Sigma$.
\end{thm}

Another application of Theorem \ref{thm} can be the normal forms for modal logics appeared in \cite{anderson} and \cite{fine}. This can be shown easily in a similar way to the one used above. Now, we will give an example of a version of first order logic whose quantifiers are not full operators, i.e. the domains of the quantifiers are not the whole set of formulas. Namely, we will construct normal forms for the so-called guarded fragment (GF) of first order logic. The guarded fragment was introduced by H. Andr\'eka, J. van Benthem and I. N\'emeti in \cite{andvannem}. What follows can be mimicked to obtain normal forms for the more liberal guarded logics, e.g. the loosely guarded fragment and the packed fragment.
\subsection{Guarded fragment of first order logic}Keep the first order language $L$ that consists of $VAR$ and $REL$. Let $\bar{v}\subseteq VAR$ be finite and suppose that $\bar{v}=\{v_0,\ldots,v_{k-1}\}$, then by $\exists\bar{v}$ we mean $\exists v_0\cdots\exists v_{k-1}$. For any first order formula $\varphi$, we write $free(\varphi)$ to mean the set of all free variables in $\varphi$. The set of \emph{GF-formulas} on $L$ is defined recursively to be the smallest subset of the set of first order formulas on $L$ that satisfies the following.
\begin{enumerate}[(a)]
\item Any first order atomic formula is a GF-formula (GF-atom).
\item If $\varphi$ and $\psi$ are GF-formulas, then $\varphi\land\psi$, $\varphi\lor\psi$ and $\lnot\varphi$ are GF-formulas.
\item\label{guards} Let $\varphi$ be a GF-formula and $G$ be an atom such that $free(\varphi)\subseteq free(G)$. Then, for any finite tuple $\bar{v}\subseteq free(G)$, $\exists\bar{v}\ (G\land\varphi)$ is a GF-formula and such $G$ is called a \emph{suitable GF-guard for $\varphi$}.
\end{enumerate}
The semantics of the guarded logic GF is the standard semantics of the classical first order logic and the set of valid formulas $\models$ is defined in the usual way. We define the normal forms for GF as follows.
\begin{defn}\label{2.1.8}Let $X'\subseteq X\subseteq VAR$ and $Y\subseteq REL$ be finite subsets. Let $At$ be the set consists of all atoms built up from $X'$ and $Y$. Assume that $At\not=\emptyset$. For any $k\in\omega$, we define the following recursively.
\begin{enumerate}[-]
\item $F_0(X,Y;X')\myeq\{At^{\alpha}:\alpha\in{^{At}\{-1,1\}}\}$.
\item $F_{k+1}(X,Y;X')\myeq\{At^{\alpha}\wedge(\overline{F_k(X,Y;X')})^{\beta}:\alpha\in{^{At}\{-1,1\}}\text{ and }\beta\in{^{\overline{F_k(X,Y;X')}}\{-1,1\}}\}$,

\noindent where $\overline{F_k(X,Y;X')}$ consists of all GF-formulas of the form $\exists \bar{u}(\gamma\land\varphi)$ such that $\gamma$ is built up from $X$ and $Y$, $\varphi\in F_k(X,Y;free(\gamma))$, $\bar{u}\subseteq free(\gamma)\subseteq X$ and $free(\gamma)\setminus\bar{u}\subseteq X'$.
\item $F(X,Y;X')\myeq\bigcup\{F_{\lambda}(X,Y;X'):\lambda\in\omega\}$.
\end{enumerate}
\end{defn}
Note that every form in $F(X,Y;X')$ is a GF-formula built up from the variables in $X$ and the relation symbols in $Y$, and whose free variables belong to the set $X'$. Before we proceed, we need to define a domain-representation system. This time we do this in a different way. Let $V=VAR$ and define $\iota(\varphi)= free(\varphi)$ for every GF-formula $\varphi$. Now, the non-propositional connectives are of the form $\exists\bar{u}(\gamma\land - )$, for some finite set of variables $\bar{u}$ and an atom $\gamma$ such that $\bar{u}\subseteq free(\gamma)$. By definition, the domain of such bounded quantifier is given by $D(\exists\bar{u}(\gamma\land - ))=\{\varphi:\varphi\text{ is GF-formula and }free(\varphi)\subseteq free(\gamma)\}$. Define $\jmath_1(\exists\bar{u}(\gamma\land - ))=\bar{u}$ and $\jmath_2(\exists\bar{u}(\gamma\land - ))= free(\gamma)$. One can easily verify that $(V,\iota,\jmath_1,\jmath_2)$ is a domain-representation system for the guarded logic $GF$.
\begin{thm}\label{GF}There is a finite algorithm that generates, for every guarded formula $\varphi$, every $k$ bigger than or equal to the maximum depth of quantifiers nesting in $\varphi$, every finite $X\subseteq VAR$ that contains all the variables appearing in $\varphi$, every $\emptyset\not=X'\subseteq X$ that contains $free(\varphi)$ and every finite $Y\subseteq REL$ that contains all relation symbols appearing in $\varphi$, a finite $\Sigma\subseteq F_k(X,Y;X')$ such that $\models\varphi\longleftrightarrow\bigvee\Sigma$.
\end{thm}
\begin{proof}
In a similar way to the one used in Theorem \ref{FOL}. Recall the domain-representation system defined above. Let $\mathbb{X}$ be the set of all atomic formulas that are constructed using the variables in $X$ and the relation symbols in $Y$. Let $\mathbb{Y}=\{\exists\bar{u}(\gamma\land - ):\bar{u}\subseteq free(\gamma)\text{ and }\gamma\in\mathbb{X}\}$. We need to show that $(k,\mathbb{X},\mathbb{Y}, X')$ is a generator suitable for $\varphi$. One can see that there is at least one relation symbol appeared in $\varphi$, this relation symbol together with the variables in $X'$ can form an atomic formula that guarantees the largeness of $X'$ for $\mathbb{X}$. Now, suppose that $\exists\bar{u}(\gamma\land - )\in Cn(\varphi)$, for some GF-guard $\gamma$ and $\bar{u}\subseteq free(\gamma)$. The guard $\gamma$ itself shows that $free(\gamma)$ is large enough for $\mathbb{X}$. It remains to note that $N_k(\mathbb{X},\mathbb{Y};X')=F_k(X,Y;X')$. Therefore, we are done by Theorem \ref{thm}.
\end{proof}
Again, the following is a corollary of Theorem \ref{GF} for the case when the language is finite, i.e. when $VAR\cup REL$ is finite.
\begin{thm}
There is a finite algorithm that generates, for every GF-formula $\varphi$ on the finite language $L$, every non-empty set of variables $X\supseteq free(\varphi)$ and every $k$ bigger than or equal to the maximum depth of quantifiers nesting in $\varphi$, a finite $\Sigma\subseteq F_k(VAR,REL;X)$ such that $\models\varphi\longleftrightarrow\bigvee\Sigma$.
\end{thm}
A slightly different version of GF, namely the solo-GF, was introduced and discussed in \cite{myphd}. The solo GF-formulas are defined analogously to the GF-formulas except that the solo GF-guarded existential quantification $\exists\bar{u}(\gamma\land - )$ is now allowed only if the block of quantifiers $\bar{u}$ is of length $\leq 1$. In \cite{myphd}, similar normal forms to the ones in Definition \ref{2.1.8} were introduced for the solo guarded fragment on a finite language. These normal forms then were used to prove that the solo GF on finite languages enjoys weak G\"odel incompleteness property but lacks G\"odel's incompleteness property (recall the fourth paragraph of the introduction). That gives more importance to the guarded logics indeed this is the first known example that differentiate the two properties for finite languages. In the following section we will show that the general normal forms for Boolean algebras with operators that was given in \cite[appendix A]{myphd} is also a special case of our general theorem.
\subsection{Boolean algebras with operators}
Let $I$ and $J$ be any two index sets and let $t=\{+,\cdot,-,0,1,f_i,d_j:i\in I\text{ and }j\in J\}$ be a similarity type such that $\{+,\cdot,-,0,1\}$ is the type of Boolean algebras, $f_i$ is an operator symbol of positive rank $n_i\myeq rank(f_i)\geq 1$ (for any $i\in I$), and $d_j$ is a constant symbol (for any $j\in J$). Let $K$ be the class of all Boolean algebras with operators of type $t$, that is the class of all algebras of type $t$, $\a{A}=\langle A,+,\cdot,-,0,1,f_i,d_j:i\in I, j\in J\rangle$, that satisfy the following:
\begin{enumerate}[(1)]
\item The Boolean part $\langle A,+,\cdot,-,0,1\rangle$ is a Boolean algebra.
\item The operators of positive ranks are additive, i.e. for any $i\in I$, any $k\in n_i$ and any $a_0,\ldots,a_{k-1},a,b,a_{k+1},\ldots,a_{n_i-1}\in A$,
    \begin{eqnarray*}
    f_i(a_0,\ldots,a_{k-1},a+b,a_{k+1},\ldots,a_{n_i-1})&=&f_i(a_0,\ldots,a_{k-1},a,a_{k+1},\ldots,a_{n_i-1})\\
    &+&f_i(a_0,\ldots,a_{k-1},b,a_{k+1},\ldots,a_{n_i-1}).
    \end{eqnarray*}
\item The operators of positive ranks are normal, i.e. for any $i\in I$, any $k\in n_i$ and any $a_0,\ldots,a_{k-1},a,a_{k+1},\ldots,a_{n_i-1}\in A$, $$a=0\implies f_i(a_0,\ldots,a_{k-1},a,a_{k+1},\ldots,a_{n_i-1})=0.$$
\end{enumerate}

Fix a set of free variables $X$ and let $T$ be the set of all terms of type $t$ generated by the variables in $X$. The set $T$ can be viewed as the set of formulas (here we call them terms) built up recursively, using the Boolean operators and the operators $(f_i:i\in I)$, from the free variables in $X$ and the constant symbols $(d_j:j\in J)$. Let $\tau,\sigma\in T$ be arbitrary terms. We write $K\models\tau=\sigma$ if, for any algebra $\a{A}\in K$ and any evaluation $\nu:X\rightarrow A$, the interpretations $[\tau]^{\a{A}}_{\nu}\in A$ and $[\sigma]^{\a{A}}_{\nu}\in A$ are equal.

Let $\Vdash$ be the set of all terms that are equal to $1$, i.e., $\Vdash=\{\tau\in T:K\models\tau=1\}$. Thus, $\langle T,\models\rangle$ is an additive logic. We note that the Boolean operators $+,\cdot$ and $-$ correspond to the propositional connectives $\lor,\land$ and $\lnot$, respectively. We also note that the equality $=$ between terms corresponds to the derived connective $\longleftrightarrow$. Let $V$ be any non-empty set. For every term $\tau\in T$, define $\iota(\tau)=V$. For every $i\in I$, define $\jmath_1(f_i)=\emptyset$ and $\jmath_2(f_i)=V$. Clearly, $(V,\iota,\jmath_1,\jmath_2)$ satisfies the conditions of Definition \ref{system}.

Let $\emptyset\not=Y\subseteq T$ be finite. Let $\prod$ and $\sum$ be the grouped versions of $\cdot$ and $+$, respectively. That is, we fix any enumeration of $Y$ and then we define $\prod Y$ and $\sum Y$ inductively according to this enumeration. The algebras in the class $K$ are Boolean algebras with operators, hence $\cdot$ and $+$ are both commutative on the elements of these algebras. Thus, it doesn't matter which enumeration we use to define $\prod$ and $\sum$, all are equal terms in the class $K$. Let $\alpha\in{{^Y}\{-1,1\}}$, define $Y^{\alpha}=\prod\{\tau^{\alpha}:\tau\in Y\}$,
where for every $\tau\in Y$, $\tau^{\alpha}=\tau$ if $\alpha(\tau)=1$ and $\tau^{\alpha}=-\tau$ otherwise. We define our normal forms as follows.

\begin{defn}Let $I'\subseteq I$, $J'\subseteq J$ and $X'\subseteq X$ be finite sets. Let $D=X'\cup\{d_j:j\in J'\}$ and suppose that $D\not=\emptyset$. For every $n\in\omega$, we define the following inductively.
\begin{enumerate}[-]
\item The normal forms of degree $0$, $F_{0}(I',J',X')\myeq\{D^{\beta}:\beta\in{^{D}\{-1,1\}}\}$.
\item The set of normal forms of degree $n+1$, $$F_{n+1}(I',J',X')\myeq\{D^{\beta}\cdot (\overline{F_n(I',J',X')})^{\alpha}:\beta\in{^{D}\{-1,1\}}\text{ and }\alpha\in{^{\overline{F_n(I',J',X')}}\{-1,1\}}\},$$

    \noindent where $\overline{F_n(I',J',X')}=\{f_i(\tau_0,\ldots,\tau_{n_i-1}):i\in I'\text{ and }\tau_0,\ldots,\tau_{n_i-1}\in F_n(I',J',X')\}$, the one-step closure of $F_n(I',J',X')$ by the operations $\langle f_i: i\in I'\rangle$.
\item The set of all forms, $F(I',J',X')\myeq\bigcup\{F_k(I',J',X'):k\in\omega\}$.
\end{enumerate}
\end{defn}
The normal forms in $F_n(I',J',X')$ are built up from the free variables in $X'$ and the constant symbols $(d_j:j\in J')$ using the operators $(f_i:i\in I')$. Moreover, every normal form is constructed based on the normal forms of the first smaller degree. Again, one can show that the following theorems are special instances of Lemma \ref{lem} and Theorem \ref{thm}.
\begin{lem}Let $I'\subseteq I$, $J'\subseteq J$ and $X'\subseteq X$ be finite sets. Let $D=X'\cup\{d_j:j\in J'\}$ and suppose that $D\not=\emptyset$. Let $n\in\omega$, then the following are true.
\begin{enumerate}[(i)]
\item\label{and1} $K\models\sum F_n(I',J',X')=1$.
\item\label{and2} For every $\tau,\sigma\in F_n(I',J',X')$, if $\tau\not=\sigma$ then $K\models\tau\cdot \sigma=0$.
\end{enumerate}
\end{lem}
\begin{thm}
There is an effective method (finite algorithm) to find, for each $\tau\in T$, each $k\in\omega$ bigger than or equal the maximum depth of non-Boolean operators nesting in $\tau$, and each finite sets $I'\subseteq I$, $J'\subseteq J$ and $X'\subseteq X$ such that $\tau$ is built up using $X'$, $(d_j:j\in J')$, the Boolean operators and $(f_i:i\in I')$, a finite $S\subseteq F_k(I',J',X')$ such that $K\models\tau=\sum S$.
\end{thm}
\begin{thm}
Suppose that $I$, $J$ and $X$ are all finite. Then there is a finite algorithm that generates, for each $\tau\in T$, an $k\in\omega$ and a finite $S\subseteq F_k(I,J,X)$ such that $K\models\tau=\sum S$.
\end{thm}
\bibliographystyle{plain}

\end{document}